\documentclass[11pt]{article}

\usepackage{tgadventor} 
\usepackage{xcolor}
\usepackage[numbers]{natbib}
\usepackage{amsthm,amsmath,mathrsfs}
\usepackage{amssymb,amsthm,amsmath}
\usepackage[dvips]{geometry}
\usepackage{amscd}
\usepackage{amsfonts}
\usepackage[latin1]{inputenc}
\usepackage[all] {xy}
\usepackage[normalem]{ulem}
\usepackage{hyperref}
\usepackage{mathtools}
\usepackage[dvips]{geometry}
\usepackage{graphicx}
\usepackage{subcaption}
\usepackage[nottoc]{tocbibind}

\newtheorem{theorem}{Theorem}[section]
\newtheorem*{maintheorem}{Main Theorem}
\newtheorem{proposition}[theorem]{Proposition}

\newtheorem{lemma}[theorem]{Lemma}
\newtheorem{definition}[theorem]{Definition}

\newtheorem{remark}[theorem]{Remark}

\newcommand{\T}{\mathbb{T}}
\newcommand{\Z}{\mathbb{Z}}
\newcommand{\N}{\mathbb{N}}
\newcommand{\R}{\mathbb{R}}
\newcommand{\C}{\mathscr{C}}

\newcommand{\To}{\longrightarrow}

\newcommand{\Fs}{\mathcal{F}^{ss}}
\newcommand{\Fu}{\mathcal{F}^{uu}}
\newcommand{\Fc}{\mathcal{F}^c}
\newcommand{\F}{\mathcal{F}}
\newcommand{\diff}{\ensuremath{\mathit{diff}(M)}}
\newcommand{\PH}{\ensuremath{\mathit{PH}(M)}}

\newcommand{\information}{{
  \bigskip
  \footnotesize
	
	\textbf{Pablo D. Carrasco: }\textsc{ICEx-UFMG, Avda. Presidente Ant\^onio Carlos 6627, Belo Horizonte-MG,BR 31270-901} \par\nopagebreak
	\textit{E-mail:} \texttt{pdcarrasco@mat.ufmg.br}

  \medskip
  \textbf{Davi Obata: }\textsc{CNRS-Laboratoire de Math\'ematiques d'Orsay, UMR 8628, Universit\'e Paris-Sud 11, Orsay Cedex 91405, France } \par\nopagebreak
  \textsc{Instituto de Matem\'atica, Universidade Federal do Rio de Janeiro, P.O. Box 68530, 21945-970, Rio de Janeiro Brazil}\par\nopagebreak
  \textit{E-mail:} \texttt{davi.obata@gmail.com}
}}

\begin{document}
\title{A new example of robustly transitive diffeomorphism}

\author{Pablo D. Carrasco\\
\and
Davi Obata\footnote{D.O. was supported by the ERC project 692925 NUHGD.}}

\maketitle

\begin{abstract}
 We present an example of a $\mathcal{C}^1$-robustly transitive skew-product with non-trivial, non-hyperbolic action on homology.
 The example is conservative, ergodic, non-uniformly hyperbolic and its fiber directions cannot be decomposed into two dominated expanded/contracted bundles.
\end{abstract}

\section{Introduction and Main Theorem}

Let $M$ be a closed Riemannian manifold and denote by \diff\ the space of $\mathcal{C}^1$ diffeomorphisms in $M$, equipped with the $\mathcal{C}^1$ topology. In trying to understand \diff, properties that are stable under perturbations play a central role in the study. This is true not only from a theoretical point of view (i.e.\@ understanding open sets in \diff), but also from an applied one, as it is desirable to maintain the same qualitative conclusions even in presence of small errors. Following common use, we will say that $f\in \diff$ has the property $\mathcal{P}$ \emph{robustly} if $\mathcal{P}$ is also valid in an open set $\mathcal{U}_f\subset \diff$ containing $f$.  

Among the robust properties that have been studied, transitivity has been one of the most extensively researched. Recall that a diffeomorphism $f$ is \textit{transitive} if for any two non-empty open sets $U$ and $V$, there is an integer $n\in \N$ such that $f^n(U) \cap V \neq \emptyset$. The first known examples of robustly transitive diffeomorphisms are given by Anosov maps \cite{anosov67, SmaleBull}: if $f\in \diff$ is transitive and  uniformly hyperbolic, then it is robustly transitive. It turns out that certain degree of hyperbolicity is required in order to have robust transitivity. Indeed, if $f\in \diff$ is robustly transitive and $dim M\leq 3$ then $f$ is hyperbolic/partially hyperbolic \cite{Mane1978,Diaz1999}; in general, $f$ admits  $Df$-invariant bundles $E,F$ such that $det(Df^{-n_0}|E),det(Df^{n_0}|D)\leq 1/2$ for some uniform $n_0\geq 1$ \cite{Bonatti2003}. It is worth to point out that the bundles $E,F$ above are not necessarily uniformly expanding \cite{Bonatti2000}.

As for non-hyperbolic examples, there are several known. The list below gives a rough (and necessarily, incomplete) picture of the arguments used to establish  robust transitivity for non hyperbolic systems.

\begin{itemize}
	\item[-] {\it Deformations from Anosov systems}. The first concrete example of non-uniformly hyperbolic robustly transitive map was given by M. Shub in \cite{Shub71};  later in \cite{Mane1978} R. Ma\~ne gave a similar type of construction on $\T^3$. They are both partially hyperbolic (see next section) and homotopic to an Anosov system, in particular with hyperbolic action on homology. The example given in \cite{Bonatti2000} is also a deformation of an Anosov diffeomorphism, and although it is not partially hyperbolic, it does admit a dominated splitting\footnote{An $f$ invariant closed set $\Lambda$ admits an dominated splitting if $T_{\Lambda}M=E\oplus F$ is an $Df$-invariant decomposition satisfying $\|Df^{n_0}(m)|E(m)\|\cdot \|Df^{-n_0}|F(f^{n_0}m)\|\leq \frac{1}{2}$ for some uniform $n_0\geq 1$.} coherent with its Anosov part (as the previous two examples). More recently, R. Potrie (\cite{PotrieThesis2012} page 152) gave an example of this type, but with the difference that it admits a dominated splitting which is not coherent with its hyperbolic part.  In these cases, the proof of robust transitivity is founded in that they have hyperbolic-type behavior in a large part of the space. 
	
	\item[-] {\it Blenders}. This powerful mechanism was introduced in \cite{Bonatti1996} by C. Bonatti and L. D\'iaz. With it the authors were able to prove that some perturbations of time-$t$ maps of mixing hyperbolic flows, and of the product of an Anosov map times the identity (say, on $\T^3$), are robustly transitive. Note that in the first case the examples are homotopic to the identity, while in the second the action on homology on the fiber direction is trivial. The same tool was used by C. Cheng, S. Gan and Y. Shi in \cite{ChengGanShi2018} to present a robustly transitive skew-product, but where the fiber action in homology is given by minus the identity. We also point out that the example in \cite{ChengGanShi2018} has some interesting ergodic properties.

	\item[-] {\it Minimality of the stable/unstable foliation}. It is easy to see that if $f\in \diff$ admits an invariant expanding minimal foliation, then $f$ is transitive. Conditions that guarantee the persistence of these types of foliations are thus relevant for robust transitivity. Among this, the property SH introduced by E. Pujals and M. Sambarino \cite{PujalsSambarino2006} is particularly simple to check, and can be applied to establish robust transitivity of transitive partially hyperbolic systems where one has some control on the behavior of the stable/unstable foliations. Shub and Ma\~ne's examples cited before fall into this category.   
	
	\item[-] {\it Non-uniform expansion along the center}. In a recent work \cite{Yang2018}, J. Yang considers partially hyperbolic systems with non-uniformly expanding center behavior, and shows that any conservative ergodic of such systems with one-dimensional center is robustly transitive. The author uses the non-uniform expanding character of the center as a replacement for hyperbolicity, employing methods of smooth ergodic theory. These techniques however seem to be applicable only for systems with one-dimensional center.  
\end{itemize}

In this note we add a different type of example to the previous list. We present a diffeomorphism that is again a partially hyperbolic skew-product on $\T^4$, but with non-hyperbolic action on homology. More importantly, the tangent bundle of the fiber neither admit any one-dimensional invariant direction, nor does it have a non-uniform expanding/contracting behavior.


Let $\T^2 = \R^2/2\pi\Z^2$ and for each $N>0$ we consider the \textit{standard map} given by $s_N(x,y) = (2N\sin x + 2x - y, x)$. Fix $A\in \mathrm{SL}(2, \Z)$ a hyperbolic matrix. On $\T^4$ we use coordinates $(x,y,z,w)$, and for each $N$ we consider the skew product $f_N: \T^2\times \T^2\To \T^2 \times \T^2$ given by\footnote{Here $[N]$ denotes the integer part of $N$.}
\[
f_N(x,y,z,w) =(s_N(x,y) + P_x \circ A^{[N]} (z,w), A^{[2N]}(z,w))\textrm{, where } P_x(x,y)=(x,0).
\]
This diffeomorphism was introduced in \cite{bergercarrasco2014} where it is proven that for large $N$ it is non-uniformly hyperbolic (i.e.\@ all its Lyapunov exponents are Lebesgue almost everywhere different from zero), and remains so by $\mathcal{C}^2$ conservative perturbations: these maps are in fact ergodic with respect to the Lebesgue measure \cite{obataerg}. It is direct to verify that the action on homology of $f_N$ is not hyperbolic, and that its fiber direction does not admit a dominated splitting (since $Df_N|\R^2\times\{0\}=Ds_N$). We remark that for a system to have the SH property, the behavior along the center has to be somehow ``homogeneous'', meaning, one has to find many points where the action on the center is expanding (or contracting) for some uniform time. The mixed behavior along the center for $f_N$ implies that it does not satisfy the SH property, and the question whether its stable/unstable foliations are (robustly) minimal seems to be outside the reach of current technology. 

Here we establish the following.

\begin{maintheorem}
There exists $N_0\in \N$ such that for any $N\geq N_0$ the diffeomorphism $f_N$ is robustly transitive (in fact, robustly topologically mixing).
\end{maintheorem}

\begin{remark}
	Topologically mixing is a stronger property than transitivity: $f$ is topologically mixing if for any two open sets $U$ and $V$, there exists $n_0\in \N$ such that for any $n\geq n_0$ we have $f^n(U) \cap V \neq \emptyset$.
\end{remark}

The proofs of robust transitivity for the diffeomorphisms which are deformations of Anosov systems, mentioned above, use information about some type of minimality (or $\varepsilon$-minimality) of stable/unstable manifolds. Observe that, our example has a hyperbolic-type behavior in a large part of the manifold, as in the examples which are deformations of Anosov systems. However, an important difference in our proof is that we do not use any information on the minimality (or $\varepsilon$-minimality) of stable/unstable foliations. Finally, we remark that in our example (and for sufficiently small $\mathcal{C}^2$ volume preserving perturbations) the manifold is the homoclinic class of an hyperbolic point, due to ergodicity and Katok's theorem \cite{Katok}. However, it remains unknown whether this is true also for $\mathcal{C}^1$ small perturbations, and we pose the question for future research.

\subsection*{Acknowledgements}
The authors thank Sylvain Crovisier for useful comments, and to the referee for her/his suggestions and pointing out several typos.



\section{Preliminaries}
\label{section.preliminaries}
In this section we present the tools we will use. We first state some general facts about partially hyperbolic diffeormorphisms and then some facts about the example we are studying.

\subsection{Partial hyperbolicity and foliations}

A diffeomorphism $f\in \diff$ is \textit{partially hyperbolic} if there exist a $Df$-invariant decomposition $TM = E^{ss}_f \oplus E^c_f \oplus E^{uu}_f $ and a Riemannian metric on $M$ such that for any $m\in M$
\begin{align*}
&\|Df(m)|_{E^{ss}_f(m)}\|<1<\|(Df(m)|_{E^{uu}_f(m)})^{-1}\|^{-1}\\
\|Df(m)|_{E^{ss}_f(m)}\|<& \|(Df(m)|_{E^{c}_f(m)})^{-1}\|^{-1} \leq  \|Df(m)|_{E^c_f(m)}\| < \|(Df(m)|_{E^{uu}_f(m)})^{-1}\|^{-1}.	
\end{align*}
The set \PH\ of partially hyperbolic diffeomorphisms is an open subset of \diff. It is well known that the distributions $E^{ss}_f$ and $E^{uu}_f$ are uniquely integrable \cite{hps}, that is, there are two unique foliations $\mathcal{F}^{ss}_f$ and $\mathcal{F}^{uu}_f$, with $C^1$-leaves, that are tangent to $E^{ss}_f$ and $E^{uu}_f$ respectively. For a point $m\in M$ we will denote by $W^{ss}_f(m)$ a leaf of the foliation $\mathcal{F}^{ss}$, we will call such leaf the strong stable manifold of $m$. Similarly, we define the strong unstable manifold of $m$ and denote it by $W^{uu}_f(m)$. We denote $E^{cs}_f = E^s_f \oplus E^c_f$ and $E^{cu}_f = E^c_f \oplus E^u_f $.


\begin{definition}
\label{dynamicalcoherence}
A partially hyperbolic diffeomorphism $f$ is \textit{dynamically coherent} if there are two invariant foliations $\mathcal{F}^{cs}_f$ and $\mathcal{F}^{cu}_f$, with $C^1$-leaves, tangent to $E^{cs}_f$ and $E^{cu}_f$ respectively. From those two foliations one obtains another invariant foliation $\mathcal{F}^c_f = \mathcal{F}^{cs}_f \cap \mathcal{F}^{cu}_f$ with $\mathcal{C}^1$ leaves that is tangent to $E^c_f$. We call these foliations the center-stable, center-unstable and center foliation.

\end{definition}

For $R>0$ we denote by $W^*_f(m;R)$ the disc of size $R$ centered on $m$, measured by the intrinsic metric in $W^*_f(m)$, for $*=ss,c,uu$. 

%

\begin{definition}
Let $f,g\in \PH$ dynamically coherent. We say that $f$ and $g$ are \textit{leaf conjugated} if there is a homeomorphism (called a leaf conjugacy) $h:M\to M$ that sends leaves of $\Fc_f$ to leaves of $\Fc_g$ and such that for any $L \in \mathcal{F}^c_f$ it is verified
\[
h(f(L)) = g(h(L)).
\]
\end{definition}

One may study the stability of partially hyperbolic systems up to leaf conjugacies. The next theorem is a good representative of this situation.



\begin{theorem}[\cite{hps}, Theorem $7.4$]
\label{leafconjugacy}
Consider $f\in \PH$ having a differentiable\footnote{More generally, the differentiability condition can be replaced by plaque expansivity. See Chapter 7 of \cite{hps}} center foliation. Then there exists an open neighborhood $U_f\subset \diff$ of $f$ such that any $g\in U_f$ is partially hyperbolic,  dynamically coherent, and leaf conjugate to $f$. The corresponding leaf conjugacy between $g$ and $f$ depends continuously on $g$.
\end{theorem}

Let $V$ and $S$ be compact manifolds. We define a partially hyperbolic skew-product as a diffeomorphism $f\in PH(V\times S)$ of the form 
\[
f(p,q)=(F_q(p),A(q)) \quad (p,q)\in V\times S,
\]
where $A:S \to S$ is a hyperbolic diffeomorphism, for each $q\in S$, the map $F_q: V \to V$ is a $C^1$-diffeomorphism and depends continuously with the choice of the point $q$, and 
\[
\|A|_{E^{ss}(q)}\| < \|(DF_q(p))^{-1}\|^{-1} \leq \|DF_q(p)\| < \|(A|_{E^{uu}(q)})^{-1}\|^{-1}, \forall (p,q)\in V\times S.
\]
In this case $f$ is dynamically coherent with center foliation $\mathcal{F}^c=\{V\times\{z\}:z\in S\}$. The example $f_N$ that we are considering is of this type.

\begin{remark}\label{continuitycoherent}
Using theorem \ref{leafconjugacy} one checks that if $f$ is a partially hyperbolic skew-product, then any diffeomorphism $g$ sufficiently close to $f$ is also partially hyperbolic, and has a center foliation $\Fc_g$ given by a trivial fibration with leaves diffeomorphic to $V$.  These leaves approach (in the Hausdorff metric) the horizontal foliation $\{V\times\{z\}:z\in S\}$ as $g\To f$. 
\end{remark}

This Remark applies in particular to $f_N$ for sufficiently large $N$. See Proposition \ref{coherence} below.

%
%

\subsection{Some estimates for the example}

Recall that for each $N\geq 0$ and $m=(x,y,z,w) \in \T^4$ we defined the diffeomorphism
\[
f_N(x,y,z,w) =(s_N(x,y) + P_x \circ A^{[N]} (z,w), A^{[2N]}(z,w))\textrm{, where }P_x(x,y)=(x,0).
\] 
Its derivative can be computed in block form
\begin{equation}\label{eq.derivative}
Df_N(m) =  
\begin{pmatrix}
Ds_N(x,y)& P_x \circ A^{[N]}\\
0& A^{[2N]}
\end{pmatrix},	
\end{equation}
where   
\begin{equation}
Ds_N(x,y) =
\begin{pmatrix}
N \cos x +2 & -1\\
1&0
\end{pmatrix}.
\end{equation}
For a point $m=(x,y,z,w)\in \T^4$ we will write $Ds_N(m)=Ds_N(x,y)$. Observe that 
\[
\frac{1}{2N} \leq \|Ds_N\| \leq 2N \textrm{ and } \|D^2s_N\| \leq N.
\]
 
Denote by $0<\lambda <1< \mu=\lambda^{-1}$ the eigenvalues of $A$, and let $e^s,e^u$  be unit eigenvectors of for $\lambda$ and $\mu$, respectively. Consider the involution $\mathcal{I}(x,y,z,w) = (y,x,z,w)$ for $(x,y,z,w) \in \T^2$. An important feature of the map $f_N$ is given by the following lemma.
\begin{lemma}[\cite{bergercarrasco2014}, Lemma $1$]
\label{lemma.involution}
The map $f_N^{-1}$ is conjugated to the map 
\[
(x,y,z,w) \mapsto (s_N(x,y) + P_x \circ A^{-[N]} (z,w), A^{-[2N]}(z,w)),
\]
by the involution $\mathcal{I}$.
\end{lemma} 
 
This lemma allows us to prove certain properties for $f_N$ and $f_N^{-1}$ only by considering the map $f_N$, since the involution tell us that $f_N$ and $f_N^{-1}$ behave in the same way up to exchanging the $x$ and $y$ coordinates. 

\begin{lemma}[\cite{bergercarrasco2014}, Corollary $5$]
\label{lemma.transversatilityunstable}
For $N$ sufficiently large, there exists a $C^1$-neighborhood $\mathcal{U}_N$ of $f_N$ such that for any $g\in \mathcal{U}_N$, for any point $m\in \T^4$ and for any unit vector $v = (v_x,v_y,v_z,v_w)$ in $E^{uu}_g(m)$, we have
\[
\lambda^N\left( \|P_x(e^u)\| - 3\lambda^N \right) \leq |v_x| \leq \lambda^N \left(\|P_x(e^u)\|+ 3\lambda^N\right).
\]
\end{lemma}
By lemma \ref{lemma.involution}, similar statement holds for the strong stable direction, but projecting on the $y$ direction.

For $m\in\T^4$, we identify $T_m\T^4=\R^4$; since the center bundle $E^c$ of $f_N$ is tangent to the horizontal fibers, by an abuse of notation we write
$E^c=\R^2 \times \{0\}=\R^2$ (the first two coordinates). We define $\pi_h:\T^2\times\T^2\To\T^2, proj_h:\R^4\To E^c$ to be the corresponding projections. Similarly, since 
the hyperbolic directions  $E^s_A$ and $E^u_A$ of $A$ on $\T^2$ are constant, by the same abuse of notation we will write $E^*_A\subset \R^4, *=s,u$ for the directions that determine on $\R^4$. If $v\in \R^4$ we write $v= (v_c,v_s, v_u)$ using the decomposition  
\begin{equation}
\label{eq.coordinatesystem}
\R^4 = E^c \oplus E^s_A \oplus E^u_A
\end{equation}
For $\alpha>0$ we define the stable cone of size $\alpha$ over $m$ by 
\[
\mathcal{C}^s_{\alpha}(m) =\{v\in T_m\T^4: v=(v_c,v_s, v_u): \|v_c+v_u\| \leq \alpha \|v_s\|\}.
\]
Note that $\mathcal{C}^s_{\alpha}=\{\mathcal{C}^s_{\alpha}(m)\}_{m\in\T^4}$ is a continuous cone field over $M$. Analogously, we define the unstable cone field  $\mathcal{C}^u_{\alpha}$ of size $\alpha$.

\begin{lemma}
\label{lemma.cones}
Fix $\alpha>0$. If $N$ is sufficiently large there exists an open neighborhood $\mathcal{U}_N$ of $f_N$ in \diff\ such that for every $g\in \mathcal{U}_N$,
the strong stable direction $E^{ss}_g$ of $g$ is contained in $\mathcal{C}^s_{\alpha}$. Similarly, the strong unstable direction $E^{uu}_g$ of $g$ is contained in $\mathcal{C}^u_{\alpha}$. 
\end{lemma}
\begin{proof}
By \eqref{eq.derivative} we deduce 
\begin{align*}
\frac{1}{2N} &\leq m(Ds_N) \leq \|Ds_N \| \leq 2N \\ 
\lambda^{N} &\leq m(P_x \circ A^N) \leq \|P_x \circ A^N \|\leq \lambda^{-N}.
\end{align*}
On the other hand, the strength of the expansion (or contraction) of $A^{2N}$ is $\lambda^{-2N}$ (respectively $\lambda^{2N}$), which is exponentially bigger than the estimates above. Therefore, a simple calculation for $N$ sufficiently large concludes the proof of the lemma for the case $g=f_N$. Noting that all bounds are stable by $\mathcal{C}^1$ perturbations we finish the proof.
\end{proof}

Lemma \ref{lemma.cones} states that for $N$ large enough, the strong stable direction is close to the stable direction of the linear Anosov $A$. Similarly, the strong unstable direction is close to the unstable direction of the linear Anosov $A$. 

Define $I=I(N)= (-2N^{-\frac{3}{10}},2 N^{-\frac{3}{10}})$ and write $C = \{\frac{\pi}{2}+I\} \cup \{\frac{3\pi}{2} +I\}$. Consider the regions 
\begin{align}
&Crit^u = C \times S^1  \times \T^2 \\
&Crit^s = S^1 \times C  \times \T^2 .
\end{align}
We define the \textit{good regions} as the sets $G^* := \T^4- Crit^*$, for $*=s,u$. For each $\theta>0$, we define the horizontal cone of size $\theta$ along the center, as 
\[
\C^{hor}_{\theta} := \{ v = (v_x,v_y)\in E^c: \|v_y\|\leq \theta \|v_x\|\}. 
\]
We define similarly the vertical cone, but exchanging the roles of $v_x$ and $v_y$ in the definition, and we denote it by $\C^{ver}_{\theta}$. Fix $\theta = N^{-\frac{3}{5}}$. 

\begin{lemma}
\label{lemma.centerconeestimates}
For every $N$ sufficiently large there exists an open neighborhood $\mathcal{U}_N$ of $f_N$ with the following property: if $g\in \mathcal{U}_N, m\in G^u$ then
\[
v\in \C^{hor}_{\theta}\Rightarrow proj_h(Dg(m) v) \subset \C^{hor}_{\theta} \textrm{ and } \|Dg(m)v\| > N^{\frac{1}{2}} \|v\|.
\] 
Furthermore, if $\gamma$ is a $C^1$-curve contained in a center leaf satisfying 
\begin{itemize}
	\item $proj_h\left(\frac{d\gamma}{dt}(t)\right)\in\C^{hor}_{\theta}\ \forall t$, and
	\item it has length greater than $N^{-\frac{3}{10}}$,
\end{itemize}	 
then the curve $g\circ \gamma$ has length greater than $4\pi$ and its horizontal projection is tangent to $\C^{hor}_{\theta}$.
\end{lemma} 

\begin{proof}
The proof follows from the proof of lemma $3.9$ and $5.5$ in \cite{obataerg}
\end{proof}


As an easy consequence of  Remark \ref{continuitycoherent}, we have the following proposition:

\begin{proposition}
\label{coherence}
Fix $\varepsilon>0$ small, for $N$ large enough there is a $C^1$-neighborhood $\mathcal{U}_N$ of $f_N$, such that if $g\in \mathcal{U}_N$ then $g$ is dynamically coherent, its center leaves are $C^1$-submanifolds, $g$ is leaf conjugated to $f_N$ and for every $m\in \T^4$ the $C^1$-distance between $W^c_g(m)$ and $W^c_{f_N}(m)$ is smaller than $\varepsilon$.
\end{proposition}

\section{Topologically mixing: proof of the Main Theorem}

\begin{lemma}
\label{lemma.goodpoint}
For every $N$ large enough, there exists a $C^1$-neighborhood $\mathcal{U}_N$ of $f_N$ such that any $g\in \mathcal{U}_N$ verifies the following properties:
\begin{enumerate}
\item If $\gamma^u\subset\Fu_g$ is non trivial curve then there exist a point $m\in \gamma^u$ and a number $n_u \geq 0$ such that $g^n(m)\in G^u$ for every $n\geq n_u$.
\item If $\gamma^s\subset\Fs_g$ is non trivial curve then there exist a point $m\in \gamma^s$ and a number $n_s \geq 0$ such that $g^{-n}(m)\in G^s$ for every $n\geq n_s$.
\end{enumerate} 
\end{lemma}

\begin{proof}
Suppose that $N$ is large enough and $\mathcal{U}_N$ is the $C^1$-open set given by lemma \ref{lemma.transversatilityunstable}. Let $g\in \mathcal{U}_N$ and $\gamma^u$ be a non trivial curve contained in a strong unstable manifold of $g$. Take $n_u\geq 0$ to be the smallest integer such that $\gamma^u_{n_u}:= g^{n_u}(\gamma^u)$ has length greater than $\lambda^{-N}\left(\|P_x(e^u)\| +3\lambda^N \right)^{-1}$.

By lemma \ref{lemma.transversatilityunstable}, we have that $\frac{l(\gamma^u_{n_u} \cap G^u)}{l(\gamma^u_{n_u})}> 1- 10 N^{-\frac{3}{10}} $. This implies that we may take $\gamma^u_{n_u+1} \subset \gamma^u_{n_u}$ a compact connected curve contained in $G^u$ with length greater than $\frac{1}{2}$. 

It is easy to see that for $N$ large enough, for any curve $\gamma$ contained in a strong unstable manifold with length greater than $\frac{1}{2}$, $g(\gamma)$ has length greater than $\lambda^{-N}\left(\|P_x(e^u)\| +3\lambda^N \right)^{-1}$. Therefore, the length of $g(\gamma^u_{n_u+1})$ is greater than $\lambda^{-N}\left(\|P_x(e^u)\| +3\lambda^N \right)^{-1}$.

Repeating this argument, we find a decreasing sequence of compact sets 
\[
\gamma^u_{n_u + 1} \supset \gamma^u_{n_u+2} \supset \cdots,
\] with the property that $g^j(\gamma^u_{n_u+n}) \subset G^u$, for $j=0, \cdots , n-1$. Take 
\[
m^u \in \displaystyle \bigcap_{n\in \N} \gamma^u_{n_u+n}.
\]
By construction, the point $m= g^{-n_u}(m^u)$ verifies the conclusion of our lemma. The argument for the stable curves is the same, working with backward iterates.

\end{proof}

Using the skew product structure of $f_N$, we can prove the following lemma:

\begin{lemma}
\label{lemma.intersectioncsu}
There exists a constant $R>0$ with the following property: for $N$ sufficiently large, there exists a $C^1$-neighborhood $\mathcal{U}$ of $f_N$ such that for any $g\in \mathcal{U}_N$ and any two points $p,q \in \T^4$ we have that for any $m_p\in W^c_g(p)$ there exists $m_q\in W^c_g(q)$ such that $W^{uu}_{g}(m_p;R) \cap W^{ss}_{g}(m_q;R) \neq \emptyset$. 
\end{lemma}
\begin{proof}
First let us prove that if $N$ is sufficiently large, we have the conclusion of the lemma for $f_N$. The robustness of this property will then follow by a transversality argument.

We consider $\pi_v:\T^4\rightarrow\T^2$ the projection on the last two coordinates, and start by noticing that (due to minimality of the foliations $\Fs_A,\Fu_A$ in $\T^2$) there is a number $R_1>0$ with the property that for any $p,q\in \T^2$ the disc $W^s_{A}(p;R_1)$ intersects transversely $W^u_{A}(q;R_1)$. By lemma \ref{lemma.cones}, there exists a constant $R_2>R_1$ such that for any point $m\in \T^4$ we have $\displaystyle{\pi_v(W^{ss}_{f_N}(m;R_2)) \supset W^s_{A}(\pi_v(m);R_1)}$.  

For any $m\in \T^4$ we consider the $C^1$-submanifolds
\begin{align*}
W^{cs}_{f_N} (m;R_2) &= \displaystyle \bigcup_{p\in W^c_{f_N}(m)} W^{ss}_{f_N}(p;R_2)\\
W^{cu}_{f_N} (m;R_2) &= \displaystyle \bigcup_{p\in W^c_{f_N}(m)} W^{uu}_{f_N}(p;R_2)
\end{align*}
By our choice of $R_2,R_1$, for any $m_1, m_2 \in \T^4$ the sets $W^{cs}_{f_N,R_2}(m_1)$ and $W^{cu}_{f_N,R_2}(m_2)$ intersect transversely; indeed their intersection is a center leaf, which shows that the conclusion of the lemma holds for $f_N$. Since the manifolds $W^{cs}_{g} (\cdot;R_2), W^{cu}_{g} (\cdot;R_2)$ depend continuously on $g$ (\cite{hps} chapter 5), the lemma folows.
\end{proof}
We fix $R$ as in lemma above and recall that $\theta = N^{-\frac{3}{5}}.$

\begin{proposition}
\label{prop.goodcurves}
If $N$ is sufficiently large there exists $\mathcal{U}_N\subset \diff$ neighborhood of $f_N$ such that for any $g\in \mathcal{U}_N$ and any open set $U\subset \T^4$, there exists a number $n_u'\geq 0$ with the following property: for every $n\geq n_u'$ there exists a $C^1$ curve $\gamma^+_n\subset g^n(U)$ satisfying:
\begin{itemize}
	\item  $\gamma^+_n$ is contained in a center leaf.
	\item  $\pi_h(\gamma^+_n)$ is tangent to $\C^{hor}_{\theta}$.
	\item  $\gamma^+_n$ has length greater than $4\pi$.
	\item  $\displaystyle \bigcup_{q\in \gamma^+_n} W^{uu}_{g}(q;R) \subset g^n(U).$ 
\end{itemize}	
Similarly, there exists $n_s'\geq 0$ such that for any $n\geq n_s'$, there exists a $C^1$ curve $\gamma^-_n \subset g^{-n}(U)$ satisfying
\begin{itemize}
	\item  $\gamma^-_n$ is contained in a center leaf.
	\item  $\pi_h(\gamma^-_n)$ is tangent to $\C^{ver}_{\theta}$.
	\item  $\gamma^-_n$ has length greater than $4\pi$
	\item  $\displaystyle \bigcup_{q\in \gamma^-_n} W^{ss}_{g}(q;R) \subset g^{-n}(U).$ 
\end{itemize}
\end{proposition}


\begin{proof}
Choose $N$ and $\mathcal{U}_N$ so that the conclusions of lemmas \ref{lemma.goodpoint} and \ref{lemma.intersectioncsu} hold. Fix $g\in \mathcal{U}_N$ and also fix two open sets $U,V\subset \T^4$. Take a small unstable curve $\gamma^u\subset U$ and consider $n_u\geq 0, m^u\in \gamma^u$  given in lemma \ref{lemma.goodpoint} (i.e.\@ $g^n(m^u)\in G^u\ \forall n\geq n_u$). Set $m^+ := g^{n_u}(m^u)$.

Since $m^+\in g^{n_u}(U)$ and the set $g^{n_u}(U)$ is open, we may take a curve 
\[\gamma^+ \subset \left(W^c(m^+) \cap g^{n_u}(U)\cap G^u \right)\]
 centered in $m^+$, such that $\pi_h(\gamma^+)$ is a horizontal segment on the torus $\T^2$. By lemma \ref{lemma.centerconeestimates}, the image $g(\gamma^+)$ projects to a curve tangent to $\C^{hor}_{\theta}$ and verifies $\mathit{length}(g(\gamma^+)) > N^{\frac{1}{2}} \mathit{length}(\gamma^+)$. The same argument as in the proof of lemmas $4.3$ and $5.7$ in \cite{obataerg} implies that there exists $n_1 \in \N$ such that for any $n\geq n_1$, there is a $C^1$ curve $\gamma^+_n \subset g^n(\gamma^+)$ with length greater than $4\pi$ and $\pi_h(\gamma^+_n)$ is tangent to $\C^{hor}_{\theta}$. 

Take $r>0$ so that 
\[
\displaystyle \bigcup_{q\in \gamma^+} W^{uu}_g(q;r) \subset g^{n_u}(U).
\]  
Fix $n_2\in \N$ such that for any $q\in \gamma^+$ and $n\geq n_2$, we have that $g^{n}(W^{uu}_g(q;r)) \supset W^{uu}_{g}(g^{n}(q;R))$. Finally, take $n_u' = \max\{n_1,n_2\}$. It follows directly that for any $n\geq n_u'$ we have
\[
\displaystyle \bigcup_{q\in \gamma^+_n} W^{uu}_{g}(q;R) \subset g^n(U), 
\]
which finishes the proof of the first part. A similar argument for $g^{-1}$ completes the proof of the proposition.
\end{proof}

Consider the vertical foliation $\F_{ver}=\{\{z\}\times\T^2:z\in\T^2\}$. Observe that if $N$ is sufficiently large and $g$ sufficiently $C^1$-close to $f_N$, we have that $W^c_g(m)$ intersects each vertical torus $\{z\}\times\T^2$ in exactly one point, for any $m\in \T^4$. Hence, for any two points $m_1, m_2\in \T^4$, the map from $W^c_g(m_1)$ to $W^c_g(m_2)$ defined by $h^g_{m_1,m_2}(p) = W^c_g(m_2) \cap \F_{ver}(p)$ is well defined. Note that $h^{f_N}_{m_1,m_2}$ is just the identity map, independently of the points $m_1, m_2$.


We recall also the notion of holonomy. For $p,q\in M$ with $q\in W^{ss}_{f_N}(p)$ define $H^{s,f_N}_{p,q}:W^c_{f_N}(p)\To W^c_{f_N}(q)$, the stable holonomy between $p$ and $q$, by
\[
H^{s,f_N}_{p,q}(w)=W^{ss}_{f_N}(w) \cap W^c_{f_N}(q),\textrm{ for $w\in W^c_{f_N}(p)$.}
\]
It is easy to see that this is a well defined map. Analogously we define $H^{s,g}_{p,q}$ for $g$ close to $f_N$. Similarly, we define the unstable holonomy map $H^{u,g}_{p,q}$ using $\Fu_g$ instead of $\Fs_g$.

Let $R>0$ be the constant given by lemma \ref{lemma.intersectioncsu}. 

\begin{lemma}
\label{lemma.c0holonomycontrol}
For every $\varepsilon>0$, there exists $N_0:= N_0(\varepsilon)$ with the following property: for $N\geq N_0$ there exists a $C^1$-neighborhood $\mathcal{U}_N$ of $f_N$ such that if $g\in \mathcal{U}_N, p\in \T^4$ and $q\in W^{ss}_{g}(p;R)$ then $d_{C^0}(h^g_{p,q}, H^{s,g}_{p,q})< \varepsilon$. Analogous result holds for the unstable holonomy. 
\end{lemma}
\begin{proof}
Fix $\varepsilon>0$. Let us first prove that the conclusion holds for $f_N$, for $N$ large. Using the coordinate system we defined in (\ref{eq.coordinatesystem}), we consider the constant vector field $X^s =  \{0\} \times e^s$, where $e^s$ is the unitary vector that generates the stable direction for the linear Anosov $A$ chosen at the beginning. Let $\{X^s_t\}_t$ be the flow  generated by $X^s$. As mentioned, since the system is a skew product, any stable manifold of $f_N$ projects to a stable manifold of $A$. In particular, for  $p\in \T^4$ and $q\in W^{ss}_{f_N}(p;R)$ there exists an unique number $T(q) \in \R$ such that $X^s_{T(q)} (.)$ is a diffeomorphism between $W^c_{f_N}(p)$ and $W^c_{f_N}(q)$. It is easy to see that  $h^{f_N}_{p,q}(m) = X^s_{T(q)}$ for $m\in W^c_{f_N}(p)$.

By lemma \ref{lemma.cones}, after fixing $\alpha \ll \frac{\varepsilon}{R}$, for $N$ large enough $E^{ss}_{f_N}$ belongs to the cone of size $\alpha$ around the direction $\{0\} \times E^s_A$. This implies that for any point $m\in \T^4$, the Hausdorff distance between the strong stable manifold $W^{ss}_{f_N}(m;R)$ and the piece of $X^s$-orbit $X^s_{[-R,R]}(m)$ is less than $\varepsilon$. By the definition of $H^{s,f_N}_{p,q}$, we conclude that $d_{C^0}(h^{f_N}_{p,q}, H^{s,f_N}_{p,q})< \varepsilon$.

Since the center leaves and compact parts of strong stable leaves vary $C^1$-continuously with the choice of a diffeomorphism $g$ in a neighborhood of $f_N$, we conclude that for any $g$ sufficiently $C^1$-close to $f_N$, $p\in \T^4$ and $q\in W^{ss}_{g}(p;R)$ we have $d_{C^0}(h^{g}_{p,q}, H^{s,g}_{p,q})< \varepsilon$.
\end{proof}

\noindent\textbf{Proof of the Main Theorem.} Fix $\varepsilon>0$ small and let $N$ be large enough with corresponding neighborhood $\mathcal{U}_N$  small enough such that the conclusions of proposition \ref{prop.goodcurves} and lemmas \ref{lemma.intersectioncsu} and \ref{lemma.c0holonomycontrol} hold. Fix $g\in \mathcal{U}_N$ and let $U,V\subset \T^4$ be any two open sets. 

\begin{figure}[h]
\centering
\includegraphics[scale= 0.6]{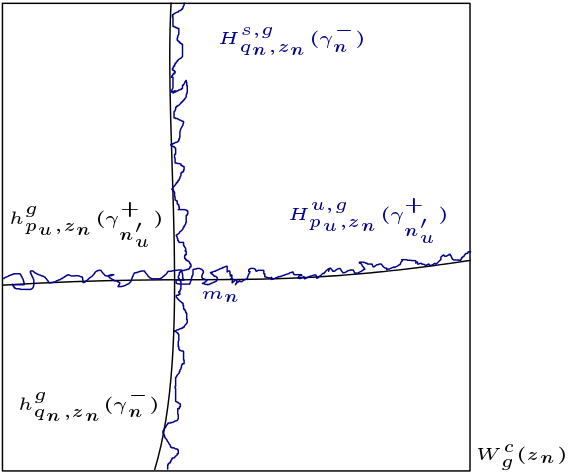}
\caption{Intersection between $H^{s,g}_{q_n,z_n}(\gamma^-_n)$ and $H^{u,g}_{p_u, z_n}(\gamma^+_{n'_u})$}
\label{figure1}
\end{figure}

By proposition \ref{prop.goodcurves} applied to $U$ for the future and $V$ for the past, we obtain two numbers $n'_u, n'_s \geq 0$ that verify the conclusion of the proposition. For $n\geq n'_s$ consider the curve $\gamma^-_n \subset W^c_{g}(q_n) \cap g^{-n}(V)$ that is almost vertical, and $\gamma^+_{n'_u}\subset W^c_g(p_u) \cap g^{n'_u}(U)$ be the almost horizontal curve given by the proposition.

Applying lemma \ref{lemma.intersectioncsu} we deduce the existence of a point $z_n\in\T^4$ such that $W^{cs}_{g}(q_n;R) \cap W^{cu}_{g}(p_u;R) = W^c_{g}(z_n)$. Observe that the image of $h^g_{q_n,z_n}(\gamma^-_n)$ is a curve $C^0$-close to a vertical curve of length $4\pi$. By lemma \ref{lemma.c0holonomycontrol}, the curve $H^{s,g}_{q_n,z_n}(\gamma^-_n)$ is also $C^0$-close to a vertical curve of length $4\pi$. Similarly, $H^{u,g}_{p_u, z_n}(\gamma^+_{n'_u})$ is a curve $C^0$-close to a horizontal curve of length $4\pi$. Therefore, the curves $H^{s,g}_{q_n,z_n}(\gamma^-_n)$ and $H^{u,g}_{p_u, z_n}(\gamma^+_{n'_u})$ must intersect at some point $m_n\in W^c_{g}(z_n)$ (see figure \ref{figure1}).

By proposition \ref{prop.goodcurves}, the point $m_n$ belongs to $g^{n'_u}(U) \cap g^{-n}(V)$. In particular, $g^{n'_u +n}(U) \cap V \neq \emptyset$. Hence, for any $n\geq n'_u + n'_s$ we have that $g^n(U) \cap V \neq \emptyset$ and $g$ is topologically mixing. This concludes the proof of the Main Theorem.


\information

\end{document}